\documentclass[12pt]{amsart}
\usepackage[utf8]{inputenc}
\usepackage{mathrsfs}
\usepackage{subfigure}
\usepackage{latexsym,amsfonts,amssymb}
\usepackage{graphicx,epsfig,subfigure}
\usepackage{psfrag}
\usepackage{enumerate}
\usepackage{colortbl}
\usepackage{verbatim}
\usepackage{hyperref}

\newcommand{\fin}{\hspace*{\fill}$\square$\vspace*{2mm}}

\makeatletter
\@namedef{subjclassname@1991}{$2020$ Mathematics Subject Classification}
\makeatother
 
\theoremstyle{plain}
\newtheorem{theorem}{Theorem}[section]
\newtheorem{lemma}[theorem]{Lemma}
\newtheorem{proposition}[theorem]{Proposition}

\theoremstyle{definition}
\newtheorem{definition}[theorem]{Definition}
\theoremstyle{remark}
\newtheorem{remark}[theorem]{\sc Remark}

\def\bC{{\mathbb C}}

\def\bN{{\mathbb N}}
\def\bP{{\mathbb P}}

\def\bR{{\mathbb R}}

\def\max{{\rm max}}

\def\const.{{\rm const.}}

\title[Lipschitz normally embedded set and tangent cones at infinity]{Lipschitz normally embedded set and tangent cones at infinity}

\author{Luis Renato G. Dias}
\address{Universidade Federal de Uberl\^andia, Faculdade de Matem\'atica, Av. Jo\~ao Naves de \'Avila 2121, 1F-153 - CEP: 38408-100, Uberl\^andia, Brazil.}
\email{lrgdias@ufu.br}

\author{Nilva Rodrigues Ribeiro}
\address{Universidade Federal de Vi\c cosa, Campus Rio Parana\'iba, Km 7, Zona Rural, MG - 230 Rodovi\'ario - CEP: 38810-000, Rio Parana\'iba, Brazil.}
\email{nilva.ribeiro@ufv.br}

\date {\today}



\thanks{The first author was partially supported by the Fapemig-Brazil Grant APQ-00431-14 and CNPq-Brazil grants 401251/2016-0 and 304163/2017-1.}

\begin{document}
	
\begin{abstract} 
We prove that any analytic set in $\bC^n$ with a unique tangent cone at infinity is an algebraic set. We prove that the degree of a complex algebraic set in $\bC^n$, which is Lipschitz normally embedded at infinity, is equal to the degree of its  tangent cone at infinity.   
\end{abstract}

\maketitle

\section{Introduction}

This note is motivated by the arguments presented in \cite{FS}. 

Let $X$ be a path connected subset  of $\bC^n$. Given  $x, y \in X$, we set  the inner distance on $X$ between $x$ and $y$, denoted by $d_X(x, y)$, as the infimum of the length($\gamma$) where  $\gamma$ varies on the set of continuous paths on $X$ connecting $x$ to $y$. We   recall the following definition from \cite{BM}. 

\begin{definition} A path connected subset $X$ of $\bC^n$ is called Lipschitz normally embedded if there exists a positive real number $C$ such that 
	\[d_X(x,y) \leq C \|x- y\|, \]
	for all $x, y \in X$.
\end{definition}

Lipschitz normal embedding and necessary conditions for a set to be Lipschitz normally embedded have been studied by many authors; see for instance \cite{BF, DTi, FS-IMRN, KPH, NPP, NPP1} and the references cited therein. 

In a recent work, Fernandes and Sampaio \cite{FS} introduced the following definition in the global case. 

\begin{definition} A path connected subset $X$ of $\bC^n$ is called Lipschitz normally embedded at infinity if there exists a compact subset $K \subset X$ 
	such that $X \setminus K$ is Lipschitz normally embedded.    
\end{definition}

When a pure dimensional analytic set $X \subset \bC^n$ has a unique  tangent cone at infinity, we denote it by $C_\infty(X)$ and call it the tangent cone of $X$ at infinity; see Definitions \ref{d:vec-inf} and \ref{d:cone-FS} for details. The main theorem of \cite{FS} is the following one: 

\begin{theorem}[{\cite[Theorem 1.1]{FS}}]\label{th:FS} 
	Let $X\subset \bC^n$ be a closed and pure $d$-dimensional analytic subset. Suppose
	$X$ has a unique tangent cone at infinity and this cone is a $d$-dimensional complex linear subspace of $\bC^n$. If $X$ is Lipschitz normally embedded at infinity, then  $X$ is an
	affine linear subspace of $\bC^n$.
\end{theorem}  

 Let $V\subset\bC^n$ be an algebraic set.  Bearing in mind that  $ \deg V=1$ if and only if $V$ is a linear subspace of $\bC^n$; see for instance Proposition 3.3 of  \cite{FS}.  We prove in this note that Theorem \ref{th:FS} holds true with the following more general  statement. 

\begin{theorem}\label{th:main} 
	Let $X\subset \bC^n$ be a closed and pure $s$-dimensional analytic subset. Suppose
	$X$ has a unique tangent cone at infinity and this cone is a pure $k$-dimensional complex algebraic  set of $\bC^n$. Then $X$ is an algebraic set and  $s = k$. Moreover, if $X$ is Lipschitz normally embedded at infinity, then  $\deg X = \deg C_{\infty}(X)$. 
\end{theorem}

We remark that from equality (2.1) of \cite{BFS}, it  follows that $\deg X \geq \deg C_{\infty}(X) $ is always true for any pure dimensional complex algebraic sets $X$ and $C_{\infty}(X)$.  

This note is organized as follows. In section 2, we recall the definition of algebraic region after \cite{Ru} and the definition of tangent cones at infinity. In  Theorem \ref{t:V-CV-same}, we prove that if an unbounded set $V\subset \bC^n$ has a unique tangent cone at infinity, we say $C_{\infty}(V)$, and this cone lies in an algebraic region, then there exists an algebraic region $\Omega$ such that $V$ and $C_{\infty}(V)$ lie in $\Omega$. 

As a consequence of Theorem \ref{t:V-CV-same}, it follows that $\dim X = \dim C_{\infty}(X)$ for any complex algebraic set (Proposition \ref{p:samedim}). This last result can be seen as a version at infinity of a result of Whitney  on tangent cones at a point of an analytic variety. 

In Section \ref{s:degree}, we recall the definition of degree of a complex  algebraic set and we present the proof of Theorem \ref{th:main}. The proof is  motivated by the arguments of \cite{FS}. 

\vspace{.2cm}
Finally, we observe that a  local version of Theorem \ref{th:main} follows from Corollary 3.13 and Remark 3.8 of \cite{FS-IMRN}. Precisely,  let $X\subset \bC^m$ be a pure analytic set and $0 \in X$. We say that $X$ is Lipschitz normally embedded at $0$ if there exists a neighborhood $U\subset \bC^m$ of $0$ such that $X\cap U$ is Lipschitz normally embedded. The tangent cone of $X$ at $0$ is denoted by $C(X,0)$; see for instance section 8.1 of  \cite{Ch}.  
We denote by $\mu_0 (X)$  and $\mu_0(C(X,0))$ the multiplicity at $0$ of $X$  and $C(X,0)$ respectively; see for instance section 11 of \cite{Ch}. Thus, a  local version of Theorem \ref{th:main} is the next result, whose proof follows from Corollary 3.13 and Remark 3.8 of \cite{FS-IMRN}.  

\begin{theorem}\label{t:local} Let $X \subset \bC^m$ be a pure dimensional complex analytic set with $0\in X$. If $X$ is a Lipschitz normally embedded set at $0$, then $\mu_0 (X) =\mu_0(C(X,0))$.  
\end{theorem}

We point out that another proof of Theorem \ref{t:local} may be given  by  similar argument of the proof  of Theorem \ref{th:main}.

\section{Algebraic region and tangent cone at infinity}\label{s:2}

\subsection{Algebraic region}

We start by recalling the definition of algebraic region after \cite{Ru}. 

\begin{definition}[{See p. 672 of \cite{Ru}}]\label{d:Ru} Let $n\geq 2$.  A set $\Omega \subset \bC^n$ is called an algebraic region of type $(k,n)$ if there exist vector subspaces $V_1, V_2$ in $\bC^n$ and positive real numbers $A, B$ such that the following conditions hold: $\dim V_1=k$, $\dim V_2=n-k$, $\bC^n = V_1 + V_2$ and $\Omega$ consists of   those $z\in \bC^n$ for which:     
	
\begin{equation}\label{eq:alg.reg.}
	 \|z''\| \leq A (1 + \|z'\|)^B,
\end{equation}
	where $z=z' + z''$, with $z'\in V_1$ and $z'' \in V_2$. 
\end{definition} 
 
A geometric characterization of analytic sets of $\bC^n$ which are algebraic is given by the next theorem. 

\begin{theorem}[See Theorem 2 of \cite{Ru}]\label{t:Ru} A closed analytic subset  $V\subset \bC^n$ of pure dimension $k$  is  algebraic  if and only if $V$ lies in some algebraic region $\Omega$ of type $(k,n)$. 
\end{theorem}

Let $\Omega \subset \bC^n$ be an algebraic region of type $(k,n)$, with $V_1, V_2\subset \bC^n$ and  constants $A, B$ as in Definition \ref{d:Ru}. We may choose linear coordinates $(x, y)$ in $\bC^n$ such that $x\in V_1$ and $y\in V_2$. In this coordinate system, we denote by $\pi_{\Omega}$ the canonical projection from  $\bC^n$ to $V_1$, i.e., $\pi_{\Omega}(x,y)=x$. In the proof of Proposition \ref{p:recip-Ru} we need a fact concerning the projection $\pi_{\Omega}$, as follows.

\begin{lemma}\label{l:1}   If $V$ is  a closed set in $\bC^n$ such that $V \subset \Omega$, then the restriction $\pi_{\Omega}: V \to V_1$ is a proper mapping.  	 
\end{lemma}
\begin{proof} 
	Let $\{z_j\}_{j\in\bN} $ be any sequence in $V$  with $\lim_{j\to \infty}\|z_j\|=\infty$. We have to show that $\lim_{j\to \infty}  \| \pi_{\Omega}(z_j)\|=\infty$.  We may write $z_j=(x_j,  y_j)$, with $x_j\in V_1$ and $y_j\in V_2$. By hypothesis, $(x_j, y_j)$ satisfies the inequality \eqref{eq:alg.reg.} and, therefore,   $\lim_{j\to \infty} \| x_j\|=\infty$. This finishes the proof, since $\pi_{\Omega}(z_j)=x_j$.
\end{proof}

Let $V\subset \bC^n$ be an algebraic set. It follows from \cite[p. 677]{Ru} that if  $\dim V \leq k$, then  $V$ lies in an algebraic region of type $(k,n)$. The converse is also true, by the next result.   

\begin{proposition}\label{p:recip-Ru} Let  $V$ be   an algebraic set  in $\bC^n$. If    
 $V$ lies in an algebraic region of type $(k,n)$, then $\dim V \leq k$. 	
\end{proposition}	
\begin{proof}  Let $\Omega \subset \bC^n$ be an algebraic region of type $(k,n)$, with $V_1, V_2$ in $\bC^n$ and  constants $A, B$ as in Definition \ref{d:Ru} and that $V\subset \Omega$.  From Lemma \ref{l:1},  the projection $\pi_{\Omega}: V \to V_1$ is a proper mapping.    Then $\dim V = \dim \pi_{\Omega}(V)$; see for instance Propositions 1 and 2 of \cite[p. 31]{Ch}. Since $\dim V_1= k$,  it follows that $\dim V \leq k$.
\end{proof}

\subsection{Tangent cone at infinity}

Let us start this section recalling two equivalent definitions of tangent cone at infinity of algebraic sets after Theorem 1 of \cite{LP}. 

\begin{definition} Let $V$ be an algebraic set. The geometric tangent cone $C_{g, \infty}(V)$ of $V$ at infinity is defined by the set of  tangent vectors of $V$ at infinity, in the sense that $v\in C_{g, \infty}(V)$ if and only if there exist sequences $\{x_j\}_{j\in\bN} \subset V$, $\{ t_j\}_{j\in\bN}\subset \bC$, such that $\lim_{j\to\infty}\|x_j\| = \infty$ and $\lim_{j\to \infty}  t_jx_j = v$. 
	
The algebraic tangent cone $C_{a, \infty}(V)$ of $V$ at infinity is defined by the set \[C_{a, \infty}(V)=\{v \in\bC^n \mid f^*(v)=0 \mbox{ for all } f \in I(V) \}, \] 	where $I(V)$ denotes the ideal defining $V$, and for each polynomial $f\in I(V)$, $f^*$ denotes its homogeneous component of highest degree. 
\end{definition} 	
 
It follows by Theorem 1 of \cite{LP} that  $C_{g, \infty}(V)=C_{a, \infty}(V)$ for any algebraic set $V\subset \bC^n$. In particular, Theorem 1 of \cite{LP} can be seen as a version at infinity of  a well-known theorem of H. Whitney \cite{Wh} on tangent cones at a point of an analytic variety.

For any unbounded subset $V$ of $\bC^m$, it is defined in \cite{FS} the tangent vectors at infinity with respect to a sequence of real positive numbers, as follows. 

\begin{definition}\label{d:vec-inf} Let $V \subset \bC^n$ be an unbounded subset. Given a sequence of real positive  numbers $\{t_j\}_{j\in \bN}$ such that $\lim_{j\to\infty} t_j=\infty$, we say that $v \in  \bC^n$ is tangent to $V$ at infinity with respect to $\{ t_j\}_{j \in \bN}$ if there exists  a sequence of points $\{ x_j\}_{j \in \bN} \subset V$ such that $ \lim_{j \to \infty}\frac{1}{t_j}x_j=v$. 
\end{definition}
	
\begin{definition}\label{d:cone-FS} Let $V \subset \bC^n$ be an unbounded subset and  $T = \{t_j\}_{j\in \bN}$ be a sequence of real positive numbers such that $\lim_{j \to \infty} t_j =\infty$. Denote  by $E_T(V)$ the set of $v \in \bC^n$ which are tangent to $V$ at infinity with respect to $T$. We call $E_T(V)$ a tangent cone of $V$ at infinity. When $V$ has a unique  tangent cone at infinity, we denote it by $C_\infty(V)$ and we call  $C_\infty(V)$ the tangent cone of $V$ at infinity.
\end{definition} 
	
In general, the problem of determining the uniqueness of tangent cones at infinity for unbounded sets is a difficult problem. For instance, a conjecture by Meeks \cite[Conjecture 3.15]{Me} states that any embedded minimal surface in $\bR^3$ with quadratic area growth has a unique tangent cone at infinity. See for instance \cite{Ga} for a partial answer of this conjecture. 

\begin{remark}\label{r:cone-alg} Let $X\subset \bC^n$ be an unbounded algebraic set. From Corollary 2.16 of \cite{FS}, we know that $X$ has an unique tangent cone at infinity and, with Theorem 1 of \cite{LP}, $C_{\infty}(X)=C_{a,\infty}(X)=C_{g, \infty}(X)$. See also Proposition 2.21 of \cite{FS}. In particular,  if $v\in C_{\infty}(X)$ then $\lambda v \in C_{\infty}(X)$, for any $\lambda \in \bR$. 
\end{remark}
 
\begin{theorem}\label{t:V-CV-same}  Let $V \subset \bC^n$ be an unbounded subset.  Suppose that 
	$V$ has a unique tangent cone at infinity and this cone lies in an algebraic region $\tilde \Omega$ of type $(k,n)$. Then there exists an algebraic region $\Omega$ of type $(k, n)$ such that $V$ and $C_{\infty}(V)$ lie in $\Omega$.   
 \end{theorem} 
 \begin{proof} It follows from Definition \ref{d:Ru} that  there exist vector spaces $V_1$ and $V_2$ in $\bC^n$, with $\dim V_1=k$, $\dim V_2=n-k$, $\bC^n= V_1 + V_2$, and positive real numbers $\tilde A, B$ such that for any $z \in \tilde \Omega$ (hence, in particular, for $z\in  C_{\infty}(V)$) we have:  
 	\begin{equation}\label{eq:3}
 		\|z''\| < \tilde  A(1+\|z'\|)^B, 
 	\end{equation}  
 	where $z=z' +z''$, with $z'\in V_1$ and $z'' \in V_2$.  Since $(1+t)^s \leq (1+t)$, for any $t\geq 0$ and $0 \leq s\leq 1$, we may assume in \eqref{eq:3} that $B\geq 1$.
 	
 	We claim that there exists a  positive real number  $R$   such that,  for any $w\in V$, we have: 
 	\begin{equation}\label{eq:2}
 		\|w''\| < R(1+\|w'\|)^{B}, 
 	\end{equation}
 	where $w=w' +w''$, with $w'\in V_1$ and $w'' \in V_2$.     
 	
 	If the claim is not true, this means that there exists a sequence   $\{ w_j \}_{j\in\mathbb{N}}  \subset V$ such that $ \| w_j''\| > k (1 + \|w_j'\|)^{B} $ and, up to a subsequence, we may suppose that   $\lim_{j \to \infty} \frac{ w_j''}{\|w_j''\|} = y_0  \in V_2$, with $\|y_0\|=1$. It follows that $\lim_{j \to \infty}\|w_j''\| =\infty$ and therefore $\lim_{j \to \infty} \|w_j\|= \infty $ since $\bC^n$ is a direct sum of $V_1$ and $V_2$.  We have 
 	
 	\[\frac{1}{k} > \frac{(1 + \|w_j'\|)^{B}}{\| w_j''\|}  \geq \frac{(1 + \|w_j'\|)}{\| w_j''\|} \geq  \frac{\|w_j'\|}{\| w_j''\|} .\] 
 	
 	Let $t_j := (\tilde A)^{-1}\|w_j''\|$. We have  $\lim_{j \to \infty} t_j =\infty$ and 
 	
 	\[ \lim_{j\to\infty} \frac{1}{t_j} w_j = \lim_{j\to\infty} \frac{1}{t_j} \left( w_j' + w_j'' \right) = \tilde A y_0.\] 
 
It follows that  $\tilde A y_0\in C_{\infty}(V)$ by Definition  \ref{d:cone-FS}. Now, since $V_2$ is a linear space and $ y_0 \in V_2$, we have $\tilde Ay_0 \in (V_2 \cap C_{\infty}(V))$. This is a contradiction, since \eqref{eq:3} does not hold for $\tilde Ay_0$.  

Therefore, if follows that  there exists a positive real number  $R$   such that  for any $w\in V$ we have: 
 	\[ 	\|w''\| < R(1+\|w'\|)^{B}, \]
 	where $z=w' +w''$, with $w'\in V_1$ and $w'' \in V_2$.   Let $ A:= \max\{R, \tilde A\}$, then $V$ and $C_{\infty}(V)$  lie in an algebraic region $\Omega$ of type $(k,n)$, with vector spaces $V_1$ and $V_2$ in $\bC^n$, with $\dim V_1=k$, $\dim V_2=n-k$  and positive real numbers $A$ and $B$. This finishes the proof. 
 \end{proof} 	

Bearing in mind that   $C_{\infty}(V)=C_{a,\infty}(V)=C_{g, \infty}(V)$  for any unbounded algebraic set $V\subset\bC^n$ (see Remark \ref{r:cone-alg}) and that, by definition of dimension, $\dim_{\bR}(V)=2\dim_{\bC} (V)$, we finish this section with the following result.

\begin{proposition}\label{p:samedim} Let $V \subset \bC^n$ be an unbounded algebraic set. Then $\dim V = \dim C_{\infty} (V)$. 
\end{proposition}
\begin{proof} It follows from Corollary 2.18 of \cite{FS} that $\dim C_{\infty}(V) \leq \dim V$.
	
On the other hand, let $k=\dim C_{\infty}(V)$, which is an algebraic set. By \cite[page 677]{Ru}, $C_{\infty}(V)$ lie in an algebraic region $\tilde \Omega$ of type $(k,n)$ and by Theorem \ref{t:V-CV-same}, $V$ and $C_{\infty}(V)$ lie in a common algebraic region $\Omega$ of type $(k, n)$. Then, it follows from Proposition \ref{p:recip-Ru} that $\dim V \leq k=\dim C_{\infty}(V)$.
\end{proof}

The previous result can be seen as a version at infinity of a result of Whitney on tangent cones at point of an analytic variety; see for instance Lemma 8.11 of \cite{Wh}.  

\section{Proof of main theorem}\label{s:degree}

We recall the definition of degree of an algebraic set. We follow section 11.3 of \cite{Ch}. The complex projective space is denoted by $\bP^n$ and $\bC^n$ is identified with the open set $\{[x_0: x_1: \ldots :x_n] \in \bP^n \mid x_0\neq 0 \}$ of $\bP^n$ through the mapping $\tau:\bC^n \to \bP^n$, $\tau(x_1, \ldots, x_n)=[1:x_1:\ldots:x_n]$.

\begin{definition} Let $X\subset \bC^n$ be an algebraic set. The degree of $X$, denoted by $\deg X$,  is the degree of its closure in $\bP^n$.   
\end{definition}

In the proof of Theorem \ref{th:main}, we consider the following.  
Let $X\subset \bC^n$ be a pure $k$-dimensional algebraic set.  We have that $X$ has a unique tangent cone at infinity and it is an algebraic set; see Remark \ref{r:cone-alg}. From Proposition \ref{p:samedim} and Theorem \ref{t:V-CV-same}, there exists an algebraic region $\Omega$ of type $(k, n)$ such that $X$ and $C_{\infty}(X)$ lie in $\Omega$. Then, there exist  vector spaces $V_1, V_2\subset \bC^n$ and  constants $A, B$ as in Definition \ref{d:Ru}. We may choose linear coordinates $(x, y)$ in $\bC^n$ such that $x\in V_1$ and $y\in V_2$. In this coordinate system, we denote by $\pi_{\Omega}$ the canonical projection from  $\bC^n$ to $V_1$, i.e., $\pi_{\Omega}(x,y)=x$. Moreover, we may assume that $V_1$ and $V_2$ are orthogonal to each other and that $\pi_{\Omega}$ is an orthogonal projection; see Theorem 3 of page 78 of \cite{Ch}.  By similar argument as in proof of Lemma \ref{l:1}, it follows that the closures of $X$ and $V_2$ (respectively, the closures of  $C_{\infty}(X)$ and $V_2$) in $\bP^n$ do not have points at infinity\footnote{here, the phrase ``at infinity'' means the set  $\bP^n \setminus \bC^n$.} in common. Then, it follows by Corollary 1 of \cite[p. 126]{Ch} the following fact concerning the projection $\pi_{\Omega}$: 

\begin{lemma}\label{l:proj} The restriction $\pi_{\Omega}: X \to V_1$ $($respectively, $\pi_{\Omega}: C_{\infty} (X) \to V_1$$)$ is a ramified covering  with number of sheets equal to $\deg X$ $($respectively, $\deg C_{\infty}(X))$. 	 
\end{lemma}
\fin

\subsection{Proof of Theorem \ref{th:main}} First, we prove that $X$ is an algebraic set and that $\dim X = \dim C_{\infty}(X)$. By Theorem \ref{t:Ru}, $C_{\infty}(X)$ lies in an algebraic region of type $(k, n)$. It follows from  Theorem \ref{t:V-CV-same} that $X$ and $C_{\infty}(X)$ lie in an algebraic region $\Omega$ of type $(k,n)$. Then,  by Theorem \ref{t:Ru}, we have that $X$ is an algebraic set. That $\dim X = \dim C_{\infty}(X)$, it follows by Proposition \ref{p:samedim}. 
 
\vspace{.2cm}

Now, we assume that $X$  is Lipschitz normally embedded at infinity. Thus, there exist a compact subset $K \subset X$ and a positive real number $C$ such that 
\begin{equation}\label{eq:p1}
	d_X(x_1, x_2) \leq C \|x_1- x_2\|, \mbox{ for all } x_1, x_2 \in X\setminus K.
\end{equation}

We denote $\deg X=d_1$ and $\deg C_{\infty}(X)=d_2$. 

From Theorem \ref{t:V-CV-same} and Proposition \ref{p:samedim}, there exists an algebraic region $\Omega $ of type $(k,n)$ containing $X$ and $C_{\infty}(X)$. We may consider $\Omega$ as in the paragraph before Lemma \ref{l:proj}. Thus, there exist  vector spaces $V_1$ and $V_2$ in $\bC^n$, with $\dim V_1=k$, $\dim V_2=n-k$, $\bC^n= V_1 + V_2$, and positive real numbers $A, B$ such that for any $z \in (X  \cup C_{\infty}(X ))$ the following holds: 
	\begin{equation}\label{eq:4}
		\|z''\| < A(1+\|z'\|)^B, 
	\end{equation}  
	where $z=z' +z''$, with $z'\in V_1$ and $z'' \in V_2$.   
	
Let $\pi_{\Omega}: \bC^n \to V_1$ be the  projection. As in the paragraph before Lemma \ref{l:proj}, we may assume that $V_1$, $V_2$  are orthogonal to each other and that $\pi_{\Omega}$ is an orthogonal projection. 

It follows by Lemma \ref{l:proj} that the restriction   $\pi_{\Omega}:X \to V_1$  (respectively, $\pi_{\Omega}:C_{\infty}(X) \to V_1$) is a ramified covering with number of sheets equal to $d_1$  (respectively, $d_2$). 

Let $\Sigma_{1}\subset V_1 $ and $\Sigma_{2}\subset V_1 $ be the ramification locus of the restriction of $\pi_{\Omega}$ to $X$ and to $C_{\infty}(X)$ respectively. 

We have that $\Sigma_{i}$ is a  complex algebraic set such that $\dim \Sigma_{i} < k$, for $i=1, 2$. Then, there exists a vector $v\in V_1 \setminus (C_{\infty}(\Sigma_{1}) \cup C_{\infty}(\Sigma_{2})) $. In particular, for big enough $t \in \bR_{>0}$, we have $t v \notin (\Sigma_{1} \cup \Sigma_2)$. 
	
Let $\lambda_1(t), \ldots,  \lambda_{d_1}(t) \in X$ be the different liftings of $tv$ by $\pi_{\Omega}$ restricted to $X$, i.e., $\pi_{\Omega}(\lambda_i(t))=tv$, $i=1, \ldots, d_1$. 

Since $\bC^n=V_1 +V_2$, we may write $\lambda_i(t)= tv +y_i(t)$, with $y_i(t) \in  V_2$, for $i=1, \ldots, d_1$.  
	
We have that $\frac{\|y_i(t)\|}{t}$ is bounded when $t\to\infty$, for $i=1, \ldots, d_1$. Otherwise, there exists an index  $i\in \{1, \ldots, d_1\}$ and there exists a sequence of positive numbers $\{t_l\}_{l \in \bN}$ such that $\frac{\|y_i(t_l)\|}{t_l} >l$, for any $l$. We may assume that  $\lim_{l\to\infty} \frac{y_i(t_l)}{\|y_i(t_l)\|}=w_i$, with $w_i \in V_2$.  It follows  that 
	\begin{equation}\label{eq:lim1}
		\lim_{l\to\infty} \frac{A}{\|y_i(t_l)\|} \lambda_i(t_l)= A\lim_{l\to\infty} \frac{t_l v +y_i(t_l)}{\|y_i(t_l)\|} = 0 +Aw_i. 
	\end{equation}
It follows by definition of $C_{\infty}(X)$ and \eqref{eq:lim1} that $A w_i \in C_{\infty}(X)$ and, therefore, $A w_i \in (C_{\infty}(X) \cap V_2)$. This contradicts $C_{\infty} (X) \subset \Omega$. 

	Therefore, for $i\in \{1, \ldots, d_1\}$, we have  that $\frac{\|y_i(t)\|}{t}$ is bounded when $t\to\infty$. Thus, we may assume that there exists a sequence of real positive numbers $\{t_j\}_{j\in\bN}$ such that $\lim_{j \to \infty} t_j =\infty$ and, for $i=1, \ldots, d_1$,  $\lim_{j \to \infty} \frac{y_i(t_j)}{t_j} = v_i$, with $v_i \in V_2$. 
	
	Then, 
	
	\begin{equation}\label{eq:d-v}
		\lim_{j\to\infty}	\frac{1}{t_j} \lambda_i(t_j) = \lim_{j\to\infty} \left(\frac{t_j v}{t_j} + \frac{y_i(t_j)}{t_j}\right)=v+ v_i, \mbox{ for } i=1, \ldots, d_1. 
	\end{equation}
	
By definition of $C_{\infty}(X)$, we have  $v+v_i \in C_{\infty}(X)$, for $i=1, \ldots, d_1$. It follows that the curves $\beta_i(t)=t(v+v_i)$ lie  in $C_{\infty}(X)$, for $i=1, \ldots, d_1$; see Remark \ref{r:cone-alg}. Therefore, $\beta_1, \ldots, \beta_{d_1}$ are liftings of $tv$ by $\pi_{\Omega}$ restricted to $C_{\infty}(X)$, i.e., $\pi_{\Omega}(\beta_i(t))=tv$, $i=1, \ldots, d_1$. 

\vspace{.2cm}
	
	If we suppose that $d_1>d_2$, then   there exists two distinct indexes  $s_1, s_2 \in \{1, \ldots, d_1\}$ such that   $\beta_{s_1}=\beta_{s_2}$, which implies $v_{s_1}=v_{s_2}$ by definitions of $\beta_{s_1}$ and $\beta_{s_2}$.  Then,   by  \eqref{eq:d-v}, it follows that  
	\begin{equation}\label{eq:to0} 
		\lim_{j\to\infty}\frac{1}{t_j} \left(\lambda_{s_1}(t_j) -\lambda_{s_2}(t_j)\right)=0. 
	\end{equation}
	
Since $v\in V_1 \setminus (C_{\infty}(\Sigma_{1}) \cup C_{\infty}(\Sigma_{2})) $, it follows from Lemma 2.4 of  \cite{LP} that there exists $\delta >0$ such that, for $j$ big enough,   
the following hold: 
\begin{equation}\label{eq:dist}
	\mathrm{dist}(t_jv, \Sigma_{1})>\delta t_j .
\end{equation}

Now, we follow Lemma 2.2 of \cite{DTi} and proof of Theorem 1.1 of \cite{FS}. For any path $\omega$ in $X$ joining $\lambda_{s_1}(t_j)$ to $\lambda_{s_2}(t_j)$, the path $\pi_{\Omega}(\omega)$ is a loop in $V_1$ with base point $t_j v$ which is not contractible in $V_1 \setminus \Sigma_1$. Then, with  \eqref{eq:dist}, it follows that the length of $\pi_{\Omega}(\omega)$  is at least $2\delta t_j$. Now,  since $\pi_{\Omega}$ is an orthogonal projection, it follows that 

\begin{equation}\label{eq:fin1}
	d_X(\lambda_{s_1} (t_j), \lambda_{s_2} (t_j)) \geq 2\delta t_j.
\end{equation} 
 
Then, by \eqref{eq:p1} and \eqref{eq:fin1}, we have:
	
\begin{equation}\label{eq:3-1}
	\frac{1}{t_j}C\|\lambda_{s_1}(t_j) - \lambda_{s_2}(t_j)\| \geq \frac{1}{t_j} d_X(\lambda_{s_1}(t_j), \lambda_{s_2}(t_j) ) \geq 2\delta, \mbox{ for $j$ big enough.}
\end{equation}
	
From \eqref{eq:to0} and \eqref{eq:3-1}, as $j\to \infty$, it follows the following contradiction: 
	
	\[ 0 \geq 2\delta.\]
Therefore, we have $d_1 \leq d_2$.  Now, as mentioned in the Introduction, $d_1 \geq d_2$ is always true for any pure dimensional complex algebraic sets $X$ and $C_{\infty}(X)$; see  equality (2.1) of \cite{BFS}. This  ends   the proof. 
\fin

\end{document}